\documentclass[11pt,reqno]{amsart}

\makeatletter

\newcommand{\Rmnum}[1]{\expandafter\@slowromancap\romannumeral #1@}
\makeatother
\usepackage{CJK}
\usepackage{colortbl}

\usepackage{latexsym,bm, amsfonts,amsthm,amsmath,mathrsfs,CJK}

 \newtheorem{lem}{Lemma}[section]  \newtheorem{thm}{Theorem}[section]
  \newtheorem{rmk}{Remark}[section] \newtheorem{example}{Example}[section]
\numberwithin{equation}{section}

\numberwithin{equation}{section}

\newcommand{\dif}{\mathrm{d}} \DeclareMathAlphabet{\mathsfsl}{OT1}{cmss}{m}{sl} \DeclareMathAlphabet{\mathpzc}{OT1}{pzc}{m}{it}

    \newcommand{\rr}{\mathbb{R}}
\newcommand{\pp}{\mathbb{P}}    

    \def\FF{\mathcal F}   \def\UU{\mathcal U}

 \def\BB{\mathscr B}

 \def\d"{^{\prime\prime}} \def\d'{^{\prime}}

\begin{document}

\title[]{Moderate deviations for empirical measures for nonhomogeneous Markov chains}
\thanks{The research is supported by Scientific Program of Department of Education of Jiang Xi Province of China (Nos. GJJ190732),  Doctor's  Scientific  Research Program of Jingdezhen Ceramic Institute (Nos. 102/01003002031 )} 
\date{} \maketitle

\begin{center}

Mingzhou Xu~\footnote{Email: mingzhouxu@whu.edu.cn}\quad Kun Cheng~\footnote{Email: chengkun0010@126.com}\\

 School of Information Engineering, Jingdezhen Ceramic Institute\\
  Jingdezhen 333403, China

\end{center}
\begin{abstract}
We prove that moderate deviations for empirical measures for countable nonhomogeneous Markov chains hold under the assumption of uniform convergence of transition probability matrices for countable nonhomogeneous Markov chains in Ces\`aro sense.
\end{abstract}

\section{Introduction and main results}
Large deviations for homogeneous Markov processes can be found in several papers of Donsker and Varadhan \cite{Donsker1975a,Donsker1975b,Donsker1976,Donsker1983}. The generalization to nonhomogeneous Markov chains was obtained by Dietz and Sethuraman \cite{DS2005}. For a survey of the theory of large deviations, the reader is referred to Varadhan \cite{Varadhan2008} and references therein. Gao \cite{Gao1992,Gao1996} studied moderate deviations for homogeneous Markov processes and martingales. Wu \cite{Wu1995,Wu2001} investigated moderate deviations and large deviations for homogeneous Markov processes. de Acosta \cite{deAcosta1997}, de Acosta and Chen \cite{deAcosta1998} obtained  moderate deviations for homogeneous Markov chains. Gao \cite{Gao2017} investigated moderate deviations for unbounded functional of homogeneous Markov processes. Central limit theorem and moderate deviation for empirical means for countable nonhomogeneous Markov chains were established by  Xu et al. \cite{Xu2019} (cf. also Xu et al. \cite{Xu2019b}). Large deviations for empirical measures for independent random variables under sublinear expectation were investigated by Gao and Xu \cite{Gao2012}.  Xu et al. \cite{Xu2020} proved moderate deviations for nonhomogeneous Markov chains with finite state space. It is interesting  to wonder whether or not moderate deviations for nonhomogeneous Markov chains with finite state space could be generalized to that for countable state space.  In this paper we try to investigate moderate deviations for empirical measures for countable nonhomogeneous Markov chains under the condition of uniform convergence of transition probability matrices for countable nonhomogeneous Markov chains in Ces\`aro sense, which complement the results of Xu et al. \cite{Xu2020}.

Assume that $\{X_n,n\ge 0\}$ defined on the probability space $(\Omega,\FF,\pp)$ is a nonhomogeneous Markov chain taking values in  $S=\{1,2,\ldots\}$ with initial probability
\begin{equation}\label{1}
\mu^{(0)}=(\mu(1),\mu(2),\ldots)
\end{equation}
and the transition matrices
\begin{equation}\label{2}
P_n=(p_n(i,j)),\mbox{  }i,j\in S, n\ge 1,
\end{equation}
where $p_n(i,j)=\pp(X_n=j|X_{n-1}=i)$. Set
$$
P^{(m,n)}=P_{m+1}P_{m+2}\cdots P_{n}, p^{(m,n)}(i,j)=\pp(X_n=j|X_m=i),  \mbox{  for $n>m$,}
$$
$$
\mu^{(k)}=\mu^{(0)}P_1P_2\cdots P_k, \mu^{(k)}(j)=\pp(X_k=j).
$$When the Markov chain is homogeneous,  we denote $P_n$, $P^{(m,m+k)}$ by $P$,  $P^k$ respectively.

If $P$ is a stochastic matrix, then we set
$$
\delta(P)=\sup_{i,k}\sum_{j=1}^{\infty}[p(i,j)-p(k,j)]^{+},
$$
where $[a]^{+}=\max\{0,a\}$.

Let $A=(a_{ij})$ be a matrix defined as $S\times S$. Set
$$
\|A\|=\sup_{i\in S}\sum_{j\in S}|a_{ij}|.
$$

If $h=(h_1,h_2,\ldots)$ , then we set $\|h\|=\sum_{j\in S}|h_j|$. If $g=(g_1,g_2,\ldots)'$ , then we set $\|g\|=\sup_{i\in S}|g_i|$. The properties below hold (see Yang \cite{Yang2002,Yang2009}):

(a)  $\|AB\|\le \|A\|\|B\|$ for all matrices $A$ and $B$;

(b) $\|P\|=1$ for all stochastic matrix $P$. 

The sequence $\{P_n,n\ge 1\}$ is said to uniformly converge in the Ces\`{a}ro sense (to a constant stochastic matrix $R$) if
\begin{equation}\label{3}
\lim_{n\rightarrow\infty}\sup_{m\ge 0}\left\|\sum_{t=1}^{n}P^{(m,m+t)}/n-R\right\|=0.
\end{equation}

For understanding notations readers could see Xu et al. \cite{Xu2020}.

 $S$ is divided into $d$ disjoint subspaces $C_0$, $C_1$, $\ldots$, $C_{d-1}$, by an irreducible stochastic matrix $P$, of period $d$ ($d\ge 1$) (see Theorem 3.3 of Hu \cite{Hu1983} ), and $P^d$ gives $d$ stochastic matrices $\{T_l,0\le l\le d-1\}$, where $T_l$ is defined on $C_l$. As in Bowerman, et al. \cite{Bowerman1977} and Yang \cite{Yang2002}, we shall discuss such an irreducible stochastic matrix $P$, of period $d$  that $T_l$ is strongly ergodic for $l=0,1,\ldots,d-1$. This matrix will be called periodic strongly ergodic.

 Throughout this paper, $E(\cdot)$ is the expectation under probability measure $\pp$, $R$ is a constant stochastic matrix each row of which is the left eigenvector $\pi=(\pi(1),\pi(2),\ldots)$ of $P$ satisfying $\pi P=\pi$ and $\sum_{i}\pi(i)=1$. By Lemma 1 in Yang \cite{Yang2009}, if (\ref{4}) below holds, then
(\ref{3}) holds.
We need the following results obtained by Xu et al. \cite{Xu2019} (cf. also Xu et al. \cite{Xu2019b}) .
\begin{thm}\label{thm1} Suppose that $\{X_n,n\ge0\}$ is a countable nonhomogeneous Markov chain taking values in $S=\{1,2,\ldots\}$ with initial distribution of (\ref{1}) and transition matrices of (\ref{2}). Assume that $f$ is a real function satisfying $|f(x)|\le M$ for all $x\in \rr$. Suppose that $P$ is a periodic strongly ergodic stochastic matrix. Assume that
\begin{equation}\label{4}
\lim_{n\rightarrow\infty}\sup_{m\ge 0}\frac1n\sum_{k=1}^{n}\|P_{k+m}-P\|=0,
\end{equation}
and
\begin{equation}\label{5}
\theta(f):=\sum_{i\in S}\pi(i)[f^2(i)-(\sum_{j\in S}f(j)p(i,j))^2]>0.
\end{equation}
Moreover, if the sequence of $\delta$-coefficient satisfies
\begin{equation}\label{6}
\lim_{n\rightarrow\infty}\frac{\sum_{k=1}^{n}\delta(P_k)}{\sqrt{n}}=0,
\end{equation}
then we have
\begin{equation}\label{7}
\frac{S_n-E(S_n)}{\sqrt{n\theta(f)}}\stackrel{D}{\Rightarrow}N(0,1)
\end{equation}
where $S_n=\sum_{k=1}^{n}f(X_k)$, $\stackrel{D}{\Rightarrow}$ stands for the convergence in distribution.
\end{thm}

Throughout this paper, we assume that
\begin{equation}\label{8}
\lim_{n\rightarrow\infty}\frac{a(n)}{\sqrt{n}}=\infty,\lim_{n\rightarrow\infty}\frac{a(n)}{n}=0.
\end{equation}
\begin{thm}\label{thm2} Under the hypotheses of Theorem \ref{thm1},
then for each open set $G\subset \rr$,
$$
\lim_{n\rightarrow\infty}\frac{n}{a^2(n)}\log\pp \left\{\frac{S_n-E(S_n)}{a(n)}\in G\right\}\ge -\inf_{x\in G}I(x),
$$
and for each closed set $F\subset \rr$,
$$
\lim_{n\rightarrow\infty}\frac{n}{a^2(n)}\log\pp \left\{\frac{S_n-E(S_n)}{a(n)}\in F\right\}\le -\inf_{x\in F}I(x),
$$
where $I(x):=\frac{x^2}{2\theta(f)}$.
\end{thm}
 $M(\rr)$ is the space of finite signed measures on $\rr$, $\tau$ denotes the topology of $\tau$-convergence on the space $M(\rr)$, i.e., $\tau$ is the topology generated by the mappings $\{\nu\in M(\rr)\mapsto \langle f,\nu\rangle, f\in \BB_ b(\rr)\}$. $M(\rr)$ is endowed with $\sigma$-fields $D_s$ generated by these mappings hereafter.  $I_{\mu}(\nu)$ in Theorem \ref{thm3} is the rate function for moderate deviations for empirical measures of countable nonhomogeneous Markov chains in Theorem \ref{thm4}. By the same method of the proof of Theorem 1.2 in Xu et al. \cite{Xu2020}, we have the following result.
\begin{thm}\label{thm3}For all $\nu\in M(\rr)$, we define
\begin{align*}
I_\mu(\nu)&:=\sup_{f\in \BB_b(\rr)}\left\{\langle f,\nu\rangle-\frac{1}{2}\sum_{i\in S}\pi(i)[f^2(i)-(\sum_{j\in S}f(j)p(i,j))^2]\right\}\\
&=\sup_{f\in \BB_b(\rr)}\left\{\langle f,\nu\rangle-\frac{1}{2}\sum_{i\in S}\sum_{j\in S}\pi(i)p(i,j)[f(j)-\sum_{k\in S}p(i,k)f(k)]^2\right\},
\end{align*}
where $\langle f,\nu\rangle=\int_{\rr}f(x)\nu(dx)$, $\BB_b(\rr)$ denotes the linear space of bounded measurable functions.
For any $l\in (0,\infty)$, $\{\nu\in M(\rr), I_{\mu}(\nu)\le l\}$ is $\tau$-compact.
\end{thm}

Set
$$
\Gamma_n(B):=\frac{1}{a(n)}\left(\delta_{X_1}(B)+\cdots+\delta_{X_n}(B)\right), B\in\BB(\rr)
$$
$$
c_n\{A\}:=\pp\left\{\Gamma_n-E(\Gamma_n)\in A\right\}, A\subset M(\rr),
$$
where $\delta_x(\cdot)$ is the dirac measure, $\BB(\rr)$ is the Borel $\sigma$-field.

Inspired by Theorems \ref{thm1}, \ref{thm2}, Our main result is the following.
\begin{thm}\label{thm4}Suppose that $\{X_n,n\ge0\}$ is a countable nonhomogeneous Markov chain taking values in $S=\{1,2,\ldots\}$ with initial distribution of (\ref{1}) and transition matrices of (\ref{2}). Let $P$ be a periodic strongly ergodic stochastic matrix. Assume that (\ref{4}) and (\ref{6}) hold.
Then $\{c_n,n\ge1\}$ satisfies moderate deviation principle in $(M(\rr), \tau)$ with speed $\frac{a^2(n)}{n}$ and with
good rate function $I_{\mu}(\cdot)$, i.e., Large deviation lower bounds
$$
\lim_{n\rightarrow \infty}\frac{n}{a^2(n)}\log c_n\{G\}\ge -\inf_{\nu\in G}I_{\mu}(\nu), \mbox{ for any open set $G\subset M(\rr)$}
$$
and Large deviation upper bounds
$$
\lim_{n\rightarrow \infty}\frac{n}{a^2(n)}\log c_n\{F\}\le -\inf_{\nu\in F}I_{\mu}(\nu), \mbox{ for any closed set $F\subset M(\rr)$}
$$
holds.
\end{thm}

We below present two examples, which are application of the above results.
\begin{example}\label{example1}We set for $\alpha>\frac12$, $P=(p(i,j))$, $p(i,j)=\frac{6}{\pi^2 j^2}$, for $i,j\ge 1$, $P_k=(p_k(i,j))$, $p_k(i,i)=\frac{6}{\pi^2 i^2}-\frac{6}{\pi^2 i^2 k^\alpha}$, for $i\ge 1$. $p_k(i,i+1)=\frac{6}{\pi^2 (i+1)^2}+\frac{6}{\pi^2 i^2 k^\alpha}$, $p_k(i,m)=\frac{6}{\pi^2 m^2}$, for $i,m\ge 1$, $|i-m|>1$. $p_k(i,i-1)=\frac{6}{\pi^2 (i-1)^2}$, for $i\ge 2$. Then $\lim_{n\rightarrow\infty}\sup_{m\ge 0}\frac1n\sum_{k=1}^{n}\|P_{k+m}-P\|=\lim_{n\rightarrow\infty}\sup_{m\ge 0,i\ge 1}\frac1n\sum_{k=1}^{n}\frac{12}{\pi^2 i^2 (k+m)^\alpha}=0$, and $\lim_{n\rightarrow\infty}\frac{\sum_{k=1}^{n}\delta(P_k)}{\sqrt{n}}=\lim_{n\rightarrow\infty}\frac{\sum_{k=1}^{n}\frac{12}{\pi^2 k^\alpha}}{\sqrt{n}}=0$, i.e. $(\ref{4})$ and $(\ref{6})$ holds.
\end{example}
\begin{example}\label{example2}We set for $\alpha>\frac12,\beta>0$,  $P=(p(i,j))$, $p(i,j)=\frac{90}{\pi^4 j^4}$, for $i,j\ge 1$, $P_k=(p_k(i,j))$, $p_k(i,i)=\frac{90}{\pi^4 i^4}-\frac{90(\log(k))^{\beta}}{\pi^4 i^4 k^\alpha}$, for $i\ge 1$. $p_k(i,i+1)=\frac{90}{\pi^4 (i+1)^4}+\frac{90(\log(k))^{\beta}}{\pi^4 i^4 k^\alpha}$, $p_k(i,m)=\frac{90}{\pi^4 m^4}$, for $i,m\ge 1$, $|i-m|>1$. $p_k(i,i-1)=\frac{90}{\pi^4 (i-1)^4}$, for $i\ge 2$. Then $\lim_{n\rightarrow\infty}\sup_{m\ge 0}\frac1n\sum_{k=1}^{n}\|P_{k+m}-P\|=\lim_{n\rightarrow\infty}\sup_{m\ge 0,i\ge 1}\frac1n\sum_{k=1}^{n}\frac{180(\log(k+m))^{\beta}}{\pi^4 i^4 (k+m)^\alpha}=0$, and $\lim_{n\rightarrow\infty}\frac{\sum_{k=1}^{n}\delta(P_k)}{\sqrt{n}}=\lim_{n\rightarrow\infty}\frac{\sum_{k=1}^{n}\frac{180(\log(k))^{\beta}}{\pi^4 k^\alpha}}{\sqrt{n}}=0$, i.e. $(\ref{4})$ and $(\ref{6})$ holds.

\end{example}

\begin{rmk}
In Example \ref{example1}, \ref{example2}, we could see a polynomial decay in $p_k(i,i)$.
\end{rmk}

In Section 2, we prove Theorem \ref{thm4}. The ideas of proofs of Theorem \ref{thm4} come from Gao \cite{Gao1996}, Gao, Xu \cite{Gao2012} and Xu et al. \cite{Xu2019,Xu2019b,Xu2020}.
\section{Proof of Theoremand \ref{thm4}}

We first establish two lemmas.
\begin{lem}\label{lem1}{\upshape[ Finite Dimensional Moderate Deviations] } For any $m\ge 1$, $f_1,\cdots,f_m\in \BB_b(\rr)$, any closed set $F\in \BB(\rr^m)$,
$$
\limsup_{n\rightarrow \infty}\frac{n}{a^2(n)}\log \pp\left\{\left(\langle f_1, \Gamma_n-E( \Gamma_n)\rangle,\cdots,\langle f_m, \Gamma_n-E( \Gamma_n)\rangle\right)\in F\right\}\le -\inf_{x\in F}I_{f,m}(x),
$$
and for any open set $G\in \BB(\rr^m)$,
$$
\liminf_{n\rightarrow \infty}\frac{n}{a^2(n)}\log \pp\left\{\left(\langle f_1, \Gamma_n-E( \Gamma_n)\rangle,\cdots,\langle f_m, \Gamma_n-E( \Gamma_n)\rangle\right)\in G\right\}\ge -\inf_{x\in G}I_{f,m}(x),
$$
where
$$
f:=(f_1,\ldots, f_m),
$$
$$I_{f,m}(x):=\sup_{z\in \rr^m}\left\{\langle x,z\rangle-\frac12\sum_{i\in S}\pi(i)\left[\left(\sum_{l=1}^{m}z_{l}f_{l}(i)\right)^2-\left(\sum_{j\in S}\sum_{l=1}^{m}z_{l}f_{l}(j)p(i,j)\right)^2\right]\right\}.$$
\end{lem}
\begin{proof}As in Gao and Xu \cite{Gao2012}, set $Y_j:=\left(f_1(X_j)-E(f_1(X_j)),\cdots,f_m(X_j)-E(f_m(X_j))\right)$.
Then by Fubini's theorem,
$$
\left(\langle f_1, \Gamma_n-E( \Gamma_n)\rangle,\cdots,\langle f_m, \Gamma_n-E( \Gamma_n)\rangle\right)=\frac1{a(n)}(Y_1+\cdots+Y_n).
$$
$Y_1$ is bounded. As in Xu et al. \cite{Xu2019}, write
$D_{n,f_i}=f_i(X_n)-E[f_i(X_n)|X_{n-1}], n\ge 1, \mbox{  } D_{0,f_i}=0$,
$W_{n,f_i}=\sum_{k=1}^{n}D_{k,f_i}$. Write $\FF_n=\sigma(X_k,0\le k\le n)$. Then $\{W_{n,f_i},\FF_n,n\ge1\}$ is a martingale, so that $\{D_{n,f_i},\FF_n,n\ge 0\}$ is the related martingale difference. For any $z=(z_1,\cdots,z_m)\in \rr^m$,
\begin{equation}\label{2.1}
\langle \sum_{i=1}^{n}Y_i,z\rangle=W_{n,\sum_{l=1}^{m}z_lf_l}+\sum_{k=1}^{n}[E[\sum_{l=1}^{m}z_lf_l(X_k)|X_{k-1}]-E[\sum_{l=1}^{m}z_lf_l(X_k)]].
\end{equation}
By (2.17) and (2.15) in Xu et al. \cite{Xu2019} (cf. also Xu et al. \cite{Xu2019b}) with $\sum_{l=1}^{m}z_lf_l$ in place of $f$, we have
\begin{equation}\label{2.2}
\lim_{n\rightarrow\infty}\frac1{\sqrt{n}} \sum_{k=1}^{n}[E[\sum_{l=1}^{m}z_lf_l(X_k)|X_{k-1}]-E[\sum_{l=1}^{m}z_lf_l(X_k)]]=0.
\end{equation}
\begin{equation}\label{2.3}
\lim_{n\rightarrow\infty}\frac{ \sum_{k=1}^{n}E(D_{k,\sum_{l=1}^{m}z_lf_l}^2)}{n}=\theta(\sum_{l=1}^{m}z_lf_l)
\end{equation}

As in the proof of (2.6) in Xu et al. \cite{Xu2020}, by (\ref{2.1})-(\ref{2.3}), we could deduce that for any $z=(z_1,\cdots,z_m)\in \rr^m$,

\begin{align*}
&\lim_{n\rightarrow\infty}\frac{n}{a^2(n)}\log E\left[\exp\left[\frac{a(n)}{n}\langle\sum_{i=1}^{n}Y_i,z\rangle\right]\right]\\
&=\lim_{n\rightarrow\infty}\frac{n}{a^2(n)}\log E\left[\exp\left[\frac{a(n)}{n} W_{n,\sum_{l=1}^{m}z_lf_l}\right]\right]\\
&=\frac{\theta(\sum_{l=1}^{m}z_lf_l)}{2}=\frac12\sum_{i\in S}\pi(i)\left[\left(\sum_{l=1}^{m}z_lf_l(i)\right)^2-\left(\sum_{j\in S}\sum_{l=1}^{m}z_lf_l(j)p(i,j)\right)^2\right],
 \end{align*}
 where $\theta(\cdot)$ is defined by (\ref{5}).
We use G\"{a}rtner-Ellis theorem (see Theorem 2.3.6 of Dembo and Zeitouni \cite{Dembo1998} or Theorem 1.4 of Yan et al. \cite{Wu1997} p. 276) to prove Lemma \ref{lem1}.
This completes the proof.
\end{proof}

\begin{lem}\label{lem2} For any $m\ge 1$, $f_1,\cdots,f_m\in \BB_b(\rr)$, define
$$
p_{f_1,\cdots,f_m}:M(\rr)\ni\nu\mapsto \left(\langle f_1,\nu\rangle,\cdots,\langle f_m,\nu\rangle\right)\in\rr^m.
$$
Then
$$
I_{f,m}(y)=\inf_{\nu=p_{f_1,\cdots,f_m}^{-1}(y)}I_{\mu}(\nu).
$$

\end{lem}
\begin{proof}
Set $\tilde{I}_{f,m}(y)=\inf_{p_{f_1,\cdots,f_m}(\nu)=y}I_{\mu}(\nu)$ and write $G_{\epsilon}(y)=\{\nu\in M(\rr);p_{f_1,\cdots,f_m}(\nu)\in B(y,\epsilon)\}$
 for $\epsilon>0$, where $B(y,\epsilon)=\{x:|x-y|<\epsilon\}$.  Then $\{\nu\in M(\rr); p_{f_1,\cdots,f_m}(\nu)=y\}$ is closed, $G_{\epsilon}(y)$ is open and $G_{\epsilon}(y)\downarrow \{\nu\in M(\rr); p_{f_1,\cdots,f_m}(\nu)=y\}$ as $\epsilon\downarrow 0$. Hence
 $$
 \tilde{I}_{f,m}(y)=\lim_{\epsilon\rightarrow 0}\inf_{\nu\in G_{\epsilon}(y)}I_{\mu}(\nu)\le \liminf_{y^{\prime} \rightarrow y}\tilde{I}_{f,m}(y^{\prime}),
 $$
 that is, $\tilde{I}_{f,m}(y)$ is lower semicontinuous. By Fenchel's theorem (cf. Strook and Deuschel \cite{STROOK2001}),
\begin{align*}
 &\frac12\sum_{i\in S}\pi(i)\left[\left(\sum_{l=1}^{m}x_lf_l(i)\right)^2-\left(\sum_{j\in S}\sum_{l=1}^{m}x_lf_l(j)p(i,j)\right)^2\right]\\
 &=\sup\left\{\int \sum_{l=1}^{m}x_lf_l(t)\nu(\dif t) -I_{\mu}(\nu) \right\}=\sup_{y\in \rr^m}\left\{\langle x,y\rangle- \tilde{I}_{f,m}(y)\right\}.
\end{align*}
Using Fenchel's theorem again, we have
$$
I_{f,m}(y)=\tilde{I}_{f,m}(y)=\inf_{p_{f_1,\cdots,f_m}(\nu)=y}I_{\mu}(\nu).
$$

\end{proof}

 By Theorem 2.7 in Yan et al. \cite{Wu1997} p. 290 and Theorem 2.1 in Yan et al. \cite{Wu1997} p. 284, we also concludes that Theorem \ref{thm4} holds from Lemmas \ref{2.1} and \ref{2.2}.

{\upshape Proof of Theorem \ref{4}:}  It remains to show that Large deviation upper bounds and lower bounds hold. For any $\nu_0\in M(\rr)$, define
$$
\UU_{\nu_0}=\left\{U_{f_1,\cdots,f_m}(\nu_0,\varepsilon):f_1,\cdots,f_m\in \BB_b(\rr),m\ge1,\varepsilon>0\right\},
$$
where $U_{f_1,\cdots,f_m}(\nu_0,\varepsilon)=\left\{\nu\in M(\rr):\left(\sum_{i=1}^{m}|\langle f_i,\nu-\nu_0\rangle|^2\right)^{1/2}<\varepsilon\right\}$.
Then $\UU_{\nu_0}$ is a basis of neighborhoods of $\nu_0$. Write $x_0=(\langle f_1,\nu_0\rangle,\cdots,\langle f_m,\nu_0\rangle)$ and $B(x_0,\varepsilon)=\{x:|x-x_0|<\varepsilon\}$. Then
\begin{align*}
&\pp\left\{\Gamma_n-E(\Gamma_n)\in U_{f_1,\cdots,f_m}(\nu_0,\varepsilon)\right\}\\
&=\pp\left\{\left(\langle f_1,\Gamma_n-E(\Gamma_n)\rangle,\cdots,\langle f_m,\Gamma_n-E(\Gamma_n)\rangle\right)\in B(x_0,\varepsilon)\right\},
 \end{align*}
 and
  \begin{align*}
  &I_{f,m}(x_0)=\sup_{z\in\rr^m}\left\{\langle \sum_{k=1}^{m}z_kf_k,\nu_0\rangle-\frac12\sum_{i\in S}\pi(i)\left[\left(\sum_{l=1}^{m}z_lf_l(i)\right)^2-\left(\sum_{j\in S}\sum_{l=1}^{m}z_lf_l(j)p(i,j)\right)^2\right]\right\}
  \end{align*}
  Obviously we see that
  \begin{align*}
  &I_{\mu}(\nu_0)=\sup_{m\ge 1}\sup_{f_1,\cdots,f_m\in \BB_b(\rr)}\sup_{z\in \rr^m}\left\{\langle \sum_{k=1}^{m}z_kf_k,\nu_0\rangle-\frac12\sum_{i\in S}\pi(i)\left[\left(\sum_{l=1}^{m}z_lf_l(i)\right)^2\right.\right.\\
  &\left.\left.-\left(\sum_{j\in S}\sum_{l=1}^{m}z_lf_l(j)p(i,j)\right)^2\right]\right\}.
  \end{align*}

  {\upshape Upper Bounds:}  For any closed set $F\subset M(\rr)$, write $L_F=\inf_{\nu\in F}I_{\mu}(\nu)$. Without loss of generality, we assume that $L_F>0$. For any
  $l\in (0,L_F)$, set $K_l=\{\nu\in M(\rr):I_{\mu}(\nu)\le l\}$ such that $K_l\bigcap F=\emptyset$. We obtain
  $$
  F=\bigcap_{f_1,\cdots,f_m\in \BB_b(\rr),m\ge 1}p_{f_1,\cdots,f_m}^{-1}\left(\overline{p_{f_1,\cdots,f_m}(F)}\right).
  $$
  By the compactness of $K_l$, there exists $f_1,\cdots,f_m\in \BB_b(\rr)$ such that
  $$p_{f_1,\cdots,f_m}^{-1}\left(\overline{p_{f_1,\cdots,f_m}(F)}\right)\bigcap K_l=\emptyset.$$
   Hence, by Lemmas \ref{2.1} and \ref{2.2},
  \begin{align*}
  &\limsup_{n\rightarrow\infty}\frac{n}{a^2(n)}\log\pp\left\{\Gamma_n-E(\Gamma_n)\in F\right\}\\&\le  \limsup_{n\rightarrow\infty}\frac{n}{a^2(n)}\log\pp \left\{\Gamma_n-E(\Gamma_n)\in p_{f_1,\cdots,f_m}^{-1}\left(\overline{p_{f_1,\cdots,f_m}(F)}\right)\right\}\\
  &\le \limsup_{n\rightarrow\infty}\frac{n}{a^2(n)}\log\pp \left\{p_{f_1,\cdots,f_m}(\Gamma_n-E(\Gamma_n))\in\overline{p_{f_1,\cdots,f_m}(F)}\right\}\\
  &\le -\inf_{x\in \overline{p_{f_1,\cdots,f_m}(F)}}I_{f,m}(x)=-\inf_{\nu\in p_{f_1,\cdots,f_m}^{-1}\overline{p_{f_1,\cdots,f_m}(F)}}I_{\mu}(\nu)\le -l.
  \end{align*}
  Letting $l\rightarrow L_F$, the proof of the large deviations upper bounds is complete.

{\upshape Lower Bounds:} For any open set $G\subset M(\rr)$, for any $\nu_0\in G$, choosing a neighborhood
$U_{f_1,\cdots,f_m}(\nu_0,\varepsilon)\subset G$. Then by Lemmas \ref{2.1} and \ref{2.2},
\begin{align*}
&\liminf_{n\rightarrow\infty}\frac{n}{a^2(n)}\log\pp\left\{\Gamma_n-E(\Gamma_n)\in G\right\}\ge \liminf_{n\rightarrow\infty}\frac{n}{a^2(n)}\log\pp\left\{\Gamma_n-E(\Gamma_n)\in U_{f_1,\cdots,f_m}(\nu_0,\varepsilon)\right\}\\
&=\liminf_{n\rightarrow\infty}\frac{n}{a^2(n)}\log\pp\left\{\left(\langle f_1,\Gamma_n-E(\Gamma_n)\rangle,\cdots,\langle f_m,\Gamma_n-E(\Gamma_n)\rangle\right)\in B(x_0,\varepsilon)\right\}\\
&\ge -I_{f,m}(x_0)\ge -I_{\mu}(\nu_0).
 \end{align*}
 This completes the proof .
 \begin{rmk}
The idea of proofs of Theorem \ref{thm4} comes from Xu et al. (\cite{Xu2019,Xu2019b}), Gao \cite{Gao1996}, and Gao and Xu \cite{Gao2012}. The key difference between the proofs of Theorem \ref{thm4} and that of Gao and Xu \cite{Gao2012} is that we could not directly use G\"artner-Ellis theorem. We need to combine Theorem 1.1 in Gao \cite{Gao1996}, the results of Xu et al. (\cite{Xu2019,Xu2019b}), and the method of Gao and Xu \cite{Gao2012} to conjecture and obtain Theorem \ref{thm4}.
 \end{rmk}

\end{document}